\documentclass[11pt]{article}
\usepackage{amsthm}
\usepackage{amsfonts}
\usepackage{amsmath}
\usepackage{colonequals}
\usepackage{epsfig}
\usepackage{url}
\newcommand{\GG}{{\cal G}}

\newtheorem{theorem}{Theorem}
\newtheorem{corollary}[theorem]{Corollary}
\newtheorem{lemma}[theorem]{Lemma}

\newtheorem{observation}[theorem]{Observation}

\newcommand{\eps}{\varepsilon}
\newcommand{\mc}[1]{\mathcal{#1}}
\newcommand{\brm}[1]{\operatorname{#1}}
\newcommand{\bb}[1]{\mathbb{#1}}

\title{Treewidth of grid subsets}
\author{Eli Berger\thanks{University of Haifa, Haifa, Israel. Supported by BSF grant no. 2006099 and by ISF grant no. 1581/12.}\and Zden\v{e}k Dvo\v{r}\'ak\thanks{Charles University, Prague, Czech Republic.
E-mail: {\tt rakdver@iuuk.mff.cuni.cz}.  Supported by the Center of Excellence -- Inst. for Theor. Comp. Sci., Prague, project P202/12/G061 of Czech Science Foundation
and by the project LL1201 (Complex Structures: Regularities in Combinatorics and Discrete Mathematics) of the Ministry of Education of Czech Republic.}\\
\and Sergey Norin\thanks{Department of Mathematics and Statistics, McGill University. Email: {\tt snorin@math.mcgill.ca}. Supported by an NSERC grant 418520.}}
\date{}

\begin{document}
\maketitle

\begin{abstract}
Let $Q_n$ be the graph of $n\times n\times n$ cube with all non-decreasing diagonals (including the facial ones) in its constituent $1\times1\times1$ cubes.
Suppose that a set $S\subseteq V(Q_n)$ separates the left side of the cube from the right side.  We show that $S$ induces
a subgraph of tree-width at least $\frac{n}{\sqrt{18}}-1$.  We use a generalization of this claim to prove that the vertex set of $Q_n$ cannot be partitioned
to two parts, each of them inducing a subgraph of bounded tree-width.
\end{abstract}

Let $G_n$ be the plane triangulated $n\times n$ grid, and consider any (non-proper) coloring of vertices of $G_n$ by two colors.
A well-known HEX lemma implies that $G_n$ contains a monochromatic path with at least $n$ vertices.
We consider a $3$-dimensional analogue of this claim.

Let $Q_n$ be the $n\times n\times n$ grid with all non-decreasing diagonals in its constituent unit cubes;
more precisely, the vertex set of $Q_n$ is $\{(x,y,z)\in\bb{Z}^3:0\le x,y,z\le n-1\}$ and two distinct vertices
$(x,y,z)$ and $(x',y',z')$ are adjacent if $x\le x'\le x+1$, $y\le y'\le y+1$ and $z\le z'\le z+1$.
A result of Matou\v{s}ek and P\v{r}\'{\i}v\v{e}tiv\'{y}~\cite{matopriv} applied in this special case shows that any coloring of vertices of $Q_n$ by two colors 
contains a connected monochromatic subgraph with
$\Omega(n^2)$ vertices.  We aim to show that it actually contains a large monochromatic subgraph which is ``2-dimensional'' in nature,
i.e., with a large grid minor.  It is well-known that a graph contains a large grid as a minor if and only if it has a large tree-width~\cite{twchu1,twchu2,RSey,quickly},
and thus we can state our main result in the following equivalent form.

\begin{theorem}\label{thm-main2}
For every $t\ge 0$, there exists $n\ge 1$ such that for any partition $A_1,A_2$ of the vertex set of $Q_n$,
either $A_1$ or $A_2$ induces a subgraph of $Q_n$ of tree-width at least $t$.
\end{theorem}

Recall that a \emph{tree decomposition} $(T,\beta)$ of a graph $G$ is a tree $T$ and a function $\beta:V(T)\to 2^{V(G)}$ assigning a \emph{bag} $\beta(u)$ to
each vertex $u\in V(T)$, such that every vertex of $G$, as well as both ends of every edge of $G$, are contained in some bag,
and such that $\{u:v\in\beta(u)\}$ induces a connected subtree of $T$ for every $v\in V(G)$.  The \emph{width} of the decomposition is
the maximum of the sizes of its bags minus one, and the \emph{tree-width} $\brm{tw}(G)$ of $G$ is the minimum of the widths of its tree decompositions.

Let us remark that the presence of diagonals is important for the validity of Theorem~\ref{thm-main2}; if the diagonals of $Q_n$
are omitted, the graph becomes bipartite, and thus it can be partitioned to two independent sets (of tree-width $0$).

Theorem~\ref{thm-main2} is motivated by a notion from the algorithmic graph theory.
We say that a class $\GG$ of graphs is \emph{tree-width fragile}
if for every $k\ge 1$, there exists $t_k\ge 0$ such that every graph $G\in\GG$,
there exist pairwise disjoint sets $A_1,\ldots,A_k\subseteq V(G)$ satisfying $\brm{tw}(G-A_i)\le t_k$ for $1\le i\le k$.
For example, planar graphs are known to have this property~\cite{baker1994approximation,rs3}.
Many interesting graph problems have efficient algorithms when restricted to graphs with bounded tree-width,
enabling Baker~\cite{baker1994approximation} to exploit this property in design of approximation algorithms for planar graphs.

It is natural to ask which more general graph classes are tree-width fragile, as the same approximation algorithms can be used for such graph classes.
Eppstein~\cite{eppstein00} proved that this is the
case for graphs avoiding some apex graph as a minor (a graph $H$ is \emph{apex} if $H-v$ is planar for some $v\in V(H)$),
and in particular for graphs embedded in any fixed surface. DeVos et al.~\cite{devospart} generalized the argument to all proper
minor-closed classes of graphs. 

Of course, not all graph classes are tree-width fragile.  It is easy to see that any graph class with
this property must be sparse and must have sublinear separators~\cite{twd}.
Theorem~\ref{thm-main2} gives another class of obstructions.
\begin{corollary}
The graph class $\{Q_n:n\ge 1\}$ is not tree-width fragile.
\end{corollary}
Indeed, Theorem~\ref{thm-main2} shows that the condition of tree-width fragility fails already for $k=2$;
unlike the previously known graph classes that are not tree-width fragile, the graphs $Q_n$ have
bounded maximum degree and balanced separators of order $O\bigl(|V(Q_n)|^{2/3}\bigr)$.
Interestingly, the class $\{Q_n:n\ge 1\}$ is \emph{fractionally} tree-width fragile, where the fractional version of
tree-width fragility is defined in the standard way~\cite{twd}; it is the first known class of graphs showing that
tree-width fragility and fractional tree-width fragility do not coincide.

Another motivation for Theorem~\ref{thm-main2} comes from graph coloring theory.  Many results are known on
the variants of the coloring where the color classes are not required to be independent sets, but rather satisfy some other
constraints, such as inducing subgraphs of bounded maximum degree~\cite{cowen1997defective,edwards2014relative}, with bounded component size~\cite{alon2003partitioning,liusmall},
or, as in our case, bounded tree-width~\cite{devospart,ding2000surfaces}.  For this notion of \emph{low tree-width coloring}, the previous results
mostly focus on positive results, showing that graphs from some class have a low tree-width coloring using a constant number of colors.
On the other hand, Theorem~\ref{thm-main2} gives a lower bound, presenting an example of a natural class of graphs that do not have low tree-width coloring
by two colors.

\bigskip

Let us now give a brief idea of the proof of Theorem~\ref{thm-main2}.  Recall that every graph containing a $t\times t$ grid
as a minor has tree-width at least $t$.  Hence, we aim to construct a monochromatic grid in $Q_n$ by connecting
appropriately chosen connected subgraphs by disjoint paths.  Of course, we need to deal with the situation that such paths are blocked
by the vertices of the other color.  Such blocking subgraphs must be in a sense larger than the subgraphs we are trying to join.
We now switch to the other color class, considering these blocking subgraphs to be the nodes of a grid we are constructing
and trying to find paths between these nodes.
We may switch back and forth between the color classes several times, repeatedly enlarging the node subgraphs.
For this procedure to end, we need to argue that eventually, the node subgraphs cannot be separated by a set inducing a subgraph
of small tree-width.  Abstracting the problem further, we reach the following claim of independent
interest.

Let $X$ be a subset of vertices of $Q_n$ that separates
the left side of the grid from the right side.  Supposing that $X$ is minimal with this property, a geometric intuition
tells us that $X$ should correspond to a surface in the $3$-dimensional Euclidean space separating the left side of $Q_n$
from the right side, and that such a surface should contain a subdivision of a large $2$-dimensional
grid in $Q_n$.  Although this geometric intuition is somewhat misleading and difficult to make precise, the overall conclusion
that $Q_n[X]$ should have large tree-width is true.

\begin{theorem}\label{thm-main1a}
Let $n\ge 1$ be an integer, let $S_1$ be the set of vertices of the left side of $Q_n$ and let $S_2$ be the set of vertices
of the right side of $Q_n$.  If a set $X\subseteq V(Q_n)\setminus (S_1\cup S_2)$ intersects every paths
in $G$ from $S_1$ to $S_2$, then the subgraph of $Q_n$ induced by $X$ has tree-width at least $\frac{n}{\sqrt{18}}-1$.
\end{theorem}

Again, note that the presence of diagonals is necessary for the validity of Theorem~\ref{thm-main1a}.
To prove Theorem~\ref{thm-main2}, we need a generalization of Theorem~\ref{thm-main1a}.
A plane graph $H$ is a \emph{near-triangulation} if every face of $H$ except for the outer one has length three.
A triple $(G,S_1,S_2)$, where $G$ is a graph and $S_1$ and $S_2$ are vertex-disjoint connected subgraphs of $G$,
is an \emph{$(n\times n)$-slab (with sides $S_1$ and $S_2$)} if there exist pairwise vertex-disjoint near-triangulations $R_1, \ldots, R_n\subseteq G$
(the \emph{rows} of the slab) and pairwise vertex-disjoint near-triangulations $C_1, \ldots, C_n\subseteq G$ (the \emph{columns} of the slab)
such that for $1\le i,j\le n$, the intersection $R_i\cap C_j$ is a path with one end in $S_1$ and the other end in $S_2$,
and such that for $s\in\{1,2\}$ and for every row or column $H$, the intersection of $H$ with $S_s$ is a subpath of the boundary of the outer face of $H$.

\begin{theorem}\label{thm-main1}
Let $(G,S_1,S_2)$ be an $(n\times n)$-slab with rows and columns of maximum degree at most $\Delta\ge 3$
and let $X$ be a subset of $V(G)\setminus (V(S_1)\cup V(S_2))$.  If every path in $G$ from $S_1$ to $S_2$
intersects $X$, then $G[X]$ has tree-width at least $\frac{n}{\sqrt{3\Delta}}-1$.
\end{theorem}

Clearly, Theorem~\ref{thm-main1} implies Theorem~\ref{thm-main1a}.
The proof of Theorem~\ref{thm-main1} is topological in nature and we give it in the following two sections.
Section~\ref{sec-main2} is devoted to the proof of Theorem~\ref{thm-main2}.

\section{ $\{0,\pm1, \star\}$-valued functions on graphs and homotopy}

In this section, we develop a discrete variant of the basic tools of homotopy theory.

Let $\mc{L}=\{0,\pm1, \star\}$. We will consider function $f: V(G)\to \mc{L}$
for a graph $G$. 
Recall that a \emph{separation} of $G$ is pair $(A,B)$ of subsets of vertices such that
$V(G)=A\cup B$ and $G$ has no edge with one end in $A\setminus B$ and the other end in $B\setminus A$.
Without explicitly assuming it, we will typically consider the
vertices with function values $-1$ and $1$ as corresponding to sets $A\setminus B$ and
$B\setminus A$, respectively, for some separation $(A,B)$ of $G$, and label $0$, as
roughly corresponding to a union of some components of $G[A \cap B]$. This
motivates the following definition. We say that $f:V(G) \to \mc{L}$ is
\emph{continuous} if vertices $u,v\in V(G)$ are not adjacent whenever $f(v)=1$
and $f(u)=-1$.  We say that $f:V(G) \to  \mc{L}$ is \emph{holomorphic} if $f$
is continuous and additionally vertices $u,v\in V(G)$ are not adjacent
whenever $f(v)=0$ and $f(u)=\star$. We say that $f:V(G) \to  \mc{L}$ is
\emph{entire}, if it is continuous and $\star \not \in \brm{Image}(f)$.
 
Let $C^0(G,\mc{L})$ denote the set of functions $f: V(G) \to \mc{L}$. Fix an
arbitrary orientation of $E(G)$, that is for every $e \in E(G)$ we distinguish
its beginning vertex denoted by $e^-$ and its end vertex denoted by $e^+$. Let
$C^1(G,\bb{Z})$ denote the lattice of all functions $f: E(G) \to \bb{Z}$. The
operator $d: C^0(G,\mc{L}) \to C^1(G,\bb{Z})$ is defined  by
$df(e)=f(e^+)-f(e^-)$ if $\star \not  \in \{f(e^+),f(e^-)\}$ and $df(e)=0$,
otherwise.

For a directed walk $W=(v_0,e_1,v_1,\ldots,e_n,v_n)$ on $G$, not necessarily respecting the orientation we fixed,
and for $1\le i\le n$, let $\eps_i\in C^1(G,\bb{Z})$ be defined by $\eps_i(e_i)=1$ if $v_i=e^+_i$, $\eps(e_i)=-1$ if
$v_i=e_i^-$ and $\eps_i(e)=0$ for $e\in E(G)\setminus \{e_i\}$.  Let $I_W=\sum_{i=1}^n\eps_i$.
For a walk $W$ and $h \in C^1(G,\bb{Z})$ we define $\int_W h \colonequals \sum_{e \in E(G)}I_W(e)h(e)$.  
Let $W^{-1}$ denote the reversal of $W$, that is the walk $(v_n,e_n,v_{n-1}, e_{n-1},\ldots, e_1,v_0)$.
For two walks $W_1$ from $u$ to $v$ and $W_2$ from $v$ to $w$, let $W_1W_2$ denote the concatenation of $W_1$ and $W_2$.
\begin{lemma}\label{lemma-entire}
If $W$ is a walk from $u$ to $v$ on a graph $G$ and $f\in C^0(G,\mc{L})$ is entire on $V(W)$ then
$$\int_W df = f(v)-f(u).$$
\end{lemma}
\begin{proof}
Let $W=(v_0,e_1,v_1,\ldots,e_n,v_n)$, where $v_0=u$ and $v_n=v$, and for $1\le i\le n$, let $\eps_i$ be as in the definition of $I_W$.
We have
\begin{align*}
\int_W df&=\sum_{e \in E(G)} I_W(e)df(e)=\sum_{i=1}^n \eps_i(e_i)df(e_i)\\
&=\sum_{i=1}^n f(v_i)-f(v_{i-1})=f(v_n)-f(v_0)=f(v)-f(u).
\end{align*}
\end{proof}
 
We refer to closed walks of length three on a graph $G$ as \emph{triangles}.
We say that a triangle $T$ is \emph{$f$-contractible} for $f \in C^0(G,\mc{L})$, if $f$ is holomorphic on $V(W)$.
The next observation is the first step towards obtaining an extension of Lemma~\ref{lemma-entire} to non-entire functions.

\begin{lemma}\label{lem:triangle}
Let $T$ be a triangle in $G$ and let $f \in C^0(G,\mc{L})$. If $f$ is continuous then $|\int_T df|\leq 1$, and if $T$ is $f$-contractible then $\int_T df=0$.
\end{lemma}
\begin{proof}
Let $V(T)=\{u,v,w\}$.  If $f$ is entire on $V(T)$, then the claim follows from Lemma~\ref{lemma-entire}.  Hence, we can assume that
$f(w)=\star$.  If $f(u)=\star$ or $f(v)=\star$, then $df(e)=0$ for every $e\in E(T)$ and $\int_T df=0$.

Otherwise, $|\int_T df|=|f(u)-f(v)|$.  Since $f$ is continuous, we have $|f(u)-f(v)|\le 1$.  Furthermore, if $T$ is $f$-contractible,
then $f(u)\neq 0\neq f(v)$, and thus $f(u)=f(v)$ and $\int_T df=0$.
\end{proof}

Let $f \in C^0(G,\mc{L})$ be continuous. From Lemma~\ref{lem:triangle} it
follows that $\int_W df=0$ for every $W$ that is a sum of $f$-contractible
triangles.  We need a slight generalization of this fact. We will say
that a closed directed walk $W$ on $G$ is \emph{$(f,k)$-almost contractible}
if there exist triangles $T_1,T_2,\ldots,T_n$ of $G$ such that
$I_W=\sum_{i=1}^n I_{T_i}$  and $T_{k+1},T_{k+2},\ldots,T_n$ are $f$-contractible.  

Let $W_i$ be a walk from $u_i$ to $v_i$ on $G$ for $i=1,2$.
We say that $W_1$ and $W_2$ are \emph{$(f,k)$-almost homotopic} if there exists a walk $Q$ from $u_1$ to $u_2$ and a walk $R$ from $v_1$ to $v_2$ such that 
the walk $QW_2R^{-1}W_1^{-1}$ is $(f,k)$-almost contractible, and furthermore $f$ is constant and 
integer on each of $V(Q)$ and $V(R)$. 
We say that $W_1$ and $W_2$ are \emph{$f$-homotopic} if they are  $(f,0)$-almost homotopic. From Lemma~\ref{lem:triangle} we deduce the following corollary. 

\begin{lemma}\label{lem:homotopic}
Let $f \in C^0(G,\mc{L})$ be continuous and let $W_1$ and $W_2$ be walks on $G$.
If $W_1$ and $W_2$ are $(f,k)$-almost homotopic then $$\left|\int_{W_1} df-\int_{W_2}df\right| \leq k.$$
In particular, if $W_1$ and $W_2$ are $f$-homotopic then $\int_{W_1} df =\int_{W_2}df$.
\end{lemma}
\begin{proof}
Let $Q$ and $R$ be the walks and $T_1,T_2,\ldots,T_n$ be the triangles showing that  $W_1$ and $W_2$ are $(f,k)$-almost homotopic.
Let $W=QW_2R^{-1}W_1^{-1}$.  By Lemma~\ref{lem:triangle}, we have
\begin{align*}
\left|\int_W df\right|&=\left|\sum_{e\in E(G)} I_W(e)df(e)\right|=\left|\sum_{e\in E(G)}\sum_{i=1}^n I_{T_i}(e)df(e)\right|\\
&=\left|\sum_{i=1}^n\int_{T_i} df\right|=\left|\sum_{i=1}^k\int_{T_i} df\right|\le k.
\end{align*}
On the other hand, since $f$ is constant on $Q$ and $R$, Lemma~\ref{lemma-entire} implies that
\begin{align*}
\left|\int_W df\right|&=\left|\int_Q df+\int_{W_1} df-\int_R df-\int_{W_2} df\right|=\left|\int_{W_1} df-\int_{W_2} df\right|,
\end{align*}
and the claim of the lemma follows.
\end{proof}

For a function $h:Y\to \bb{R}$ and a set $X\subseteq Y$, let $h(X)=\sum_{x\in X} h(x)$.
Let $P=v_0v_1\ldots v_n$ be a path in a graph $G$.  For a function $f\in C^0(G,\mc{L})$ that is entire on $P$ and for $1\le i\le n-1$, let
$\lambda_{P,f}(v_i)=\frac{1}{2}(f(v_{i+1})-f(v_{i-1}))$.

\begin{lemma}\label{lemma-path}
Let $f\in C^0(G,\mc{L})$ be a function that is entire on a path $P=v_0v_1\ldots v_n$ in a graph $G$, such that $f(v_0)\neq 0$ and $f(v_n)\neq 0$.
Suppose that $g\in C^0(G,\mc{L})$ satisfies $g(v)=f(v)$ for every $v\in V(P)$ such that $g(v)\neq \star$, and $f(v)=0$ for every  $v\in V(P)$ such that $g(v)=\star$.
Let $X=\{v\in V(P):g(v)=0\}$.  Then
$$\frac{1}{2}\int_P dg=\lambda_{P,f}(X).$$  Furthermore, if $g$ is holomorphic on $P$, then $\lambda_{P,f}(X)$ is an integer.
\end{lemma}
\begin{proof}
Let $P_1$, \ldots, $P_k$ be the maximal subpaths of $P$ such that $g$ is entire on $P_i$ for $i=1,\ldots,k$.
Note that $$\int_P dg=\sum_{j=1}^k \int_{P_j} dg,$$
and thus it suffices to prove that $\frac{1}{2}\int_{P_j} dg=\lambda_{P,f}(X\cap V(P_j))$ for $1\le j\le k$.
Let $P_j=v_av_{a+1}\ldots v_b$.  By Lemma~\ref{lemma-entire}, we have $\frac{1}{2}\int_{P_j} dg=\frac{1}{2}(g(v_b)-g(v_a))=\frac{1}{2}(f(v_b)-f(v_a))$.
Let $V(P_j)\cap X=\{v_{c_1},v_{c_2},\ldots, v_{c_p}\}$, where $c_1<c_2<\ldots<c_p$.
Since $f$ is continuous on $P$, for $1\le i\le p-1$ we have $f(v_{c_i+1})=f(v_{c_{i+1}-1})$.
Furthermore, $f(v_{c_1-1})=f(v_a)$ and $f(v_{c_p+1})=f(v_b)$; this is the case even if say $c_1=a$, since then
$g(v_{a-1})=\star$ and $f(v_{a-1})=f(v_a)=0$.
Therefore, $\lambda_{P,f}(X\cap V(P_j))=\frac{1}{2}\sum_{i=1}^p f(v_{c_i+1})-f(v_{c_i-1})=\frac{1}{2}(f(v_b)-f(v_a))=\frac{1}{2}\int_{P_j}dg$, as required.
Moreover, if $g$ is holomorphic on $P$, then $f(v_a)=\pm 1$, since either $a=0$ or $g(v_{a-1})=\star$, and similarly $f(v_b)=\pm 1$.
Hence, $\int_{P_j} dg\in\{-2,0,2\}$ and $\lambda_{P,f}(X\cap V(P_j))=\frac{1}{2} \int_{P_j} dg$ is an integer.
\end{proof}

\section{Treewidth of slab separators}

To give a lower bound on tree-width, we use the following claim (which is well-known, although usually stated with
only non-negative weights).

\begin{lemma}\label{lem:weights}
Let $t\ge 0$ be an integer.
Let $H$ be a graph with $\brm{tw}(H) \leq t$.  Let $\lambda:V(H) \to \bb{R}$ be such that $|\lambda(v)|\le 1$ for every $v\in V(H)$.
If $\lambda(V(H)) \geq 3t+3$, then there exists
a separation $(K,L)$ of $H$ such that $$\frac{1}{3}\lambda(V(H)) \leq \lambda(K\setminus L) \leq \frac{2}{3}\lambda(V(H))$$ and $|K \cap L| \leq t+1$.
\end{lemma}
\begin{proof}
Let $(T,\beta)$ be a tree decomposition of $H$ with bags of size at most $t+1$.
For an edge $uv$ of $T$, let $T_{u,v}$ denote
the component of $T-uv$ containing $v$, and let $S_{u,v}=\bigcup_{w\in V(T_{u,v})} \beta(w)\setminus \beta(u)$.

We first show that there exists $u\in V(T)$ such that every neighbor $v$ of $u$ in $T$ satisfies
$\lambda(S_{u,v})\le \frac{2}{3}\lambda(V(H))$.  If not, then for each $u\in V(T)$, let
$\pi(u)$ denote a neighbor of $u$ in $T$ such $\lambda(S_{u,\pi(u)})>\frac{2}{3}\lambda(V(H))$.
Since $T$ is a tree, there exists an edge $uv\in E(T)$ such that $\pi(u)=v$ and $\pi(v)=u$.
However, then
$$\lambda(V(H))=\lambda(S_{u,v})+\lambda(S_{v,u})+\lambda(\beta(u)\cap \beta(v))>\frac{4}{3}\lambda(V(H))-(t+1),$$
which contradicts the assumption $\lambda(V(H)) \geq 3t+3$.

Let $u\in V(T)$ be a vertex such that every neighbor $v$ of $T$ satisfies
$\lambda(S_{u,v})\le \frac{2}{3}\lambda(V(H))$.  Let $v_1$, $v_2$, \ldots, $v_m$ be the neighbors of $u$ in $T$
ordered so that $\lambda(S_{u,v_1})\ge \lambda(S_{u,v_2})\ge\ldots\ge \lambda(S_{u,v_m})$.
Note that $\sum_{i=1}^m S_{u,v_i}=\lambda(V(H))-\lambda(\beta(u))\ge \lambda(V(H))-(t+1)\ge\frac{1}{3}\lambda(V(H))$.
Let $m'$ be the smallest index such that $$\sum_{i=1}^{m'} \lambda(S_{u,v_i})\ge \frac{1}{3}\lambda(V(H)).$$
Let $K=\beta(u)\cup \bigcup_{i=1}^{m'} S_{u,v_i}$ and $L=\beta(u)\cup \bigcup_{i=m'+1}^{m} S_{u,v_i}$.
By the choice of $m'$, we have $\lambda(K\setminus L)\ge \frac{1}{3}\lambda(V(H))$.  If $m'=1$, then
$\lambda(K\setminus L)=\lambda(S_{u,v_1})\le \frac{2}{3}\lambda(V(H))$.
Hence, we can assume that $m'\ge 2$, and thus $\lambda(S_{u,v_i})<\frac{1}{3}\lambda(V(H))$ for $1\le i\le m$.
By the choice of $m'$, we have $\bigcup_{i=1}^{m'-1} \lambda(S_{u,v_i})<\frac{1}{3}\lambda(V(H))$,
and thus $\lambda(K\setminus L)<\frac{1}{3}\lambda(V(H))+\lambda(S_{u,v_{m'}})<\frac{2}{3}\lambda(V(H))$.
\end{proof}

We are now ready to prove our first main result.

\begin{proof}[Proof of Theorem~\ref{thm-main1}.]
Let $R_1$, \ldots, $R_n$ and $C_1$, \ldots, $C_n$ be rows and columns of $G$ of maximum degree at most $\Delta$.
For $1\le i,j\le n$, let $G\uparrow(i,j)$ denote the path $R_i\cap C_j$, directed from $S_1$ to $S_2$.

Let $(A,B)$ be a separation of $G$ such that $V(S_1)\subseteq A\setminus B$, $V(S_2)\subseteq B\setminus A$ and $X=A\cap B$.
Let $f:V(G) \to \mc{L}$ be defined by 
$$f(v)=\begin{cases} -1, & \text{if } v \in A\setminus B, \\
 0, &  \text{if } v \in X,\\
 1, &  \text{if } v \in B \setminus A.
\end{cases}$$
Note that $f$ is entire.  By Lemma~\ref{lemma-entire} we have
\begin{equation}\label{eq:1}
\int_{G\uparrow(i,j)}df=2.
\end{equation}
for $1\le i,j\le n$.

For every $v\in V(G\uparrow(i,j))\cap X$, let $\lambda(v)=\lambda_{G\uparrow(i,j),f}(v_i)$.
For every $v\in X\setminus\bigcup_{1\le i,j\le n} V(G\uparrow(i,j))$, let $\lambda(v)=0$.
By Lemma~\ref{lemma-path},
\begin{equation}\label{eq:2}
\int_{G\uparrow(i,j)}df=2\lambda(V(G \uparrow (i,j)) \cap X)
\end{equation}
It follows from (\ref{eq:1}) and (\ref{eq:2}) that $\lambda(X)= n^2$.

Let $H\colonequals G[X]$ and let $t\colonequals \brm{tw}(H)$.  We aim to prove that $t\ge \frac{n}{\sqrt{3\Delta}}-1$.
If $t>n^2/3-1$ or $t\ge n-1$, then this claim holds, since $\Delta\ge 3$.  Hence, we can assume that $n^2\ge 3t+3$ and $t<n-1$.
By Lemma~\ref{lem:weights} there exists a separation $(K,L)$ of $H$ such that 
\begin{equation}\label{eq:2.5}
\frac{n^2}{3} \leq \lambda(K\setminus L) \leq \frac{2n^2}{3}
\end{equation}
and $|K \cap L| \leq t+1<n$.
Let $R\colonequals \{i \in [n] : V(R_i)\cap (K\cap L)=\emptyset\}$ and
$C\colonequals \{j \in [n] : V(C_j)\cap (K\cap L)=\emptyset\}$, be the sets of indices of rows and, respectively, columns of $(G,S_1,S_2)$ disjoint from $K\cap L$.
Note that these sets are non-empty.
Let $S\colonequals (R \times [n]) \cup ([n] \times C)$.

Let $g \in C^0(G,\mc{L})$ be 
defined by $g(v)=f(v)$ for $v \in V(G)\setminus L$ and $g(v)=\star$ for $v \in L$. Note that $g$ is continuous and holomorphic on $V(G) - (K\cap L)$.
Let $h(p)\colonequals \frac{1}{2}\int_{G\uparrow p}dg$ for $p \in [n]^2$.
By Lemma~\ref{lemma-path}, we have
\begin{equation}\label{eq:4}
h(p) = \lambda(V(G \uparrow p) \cap (K\setminus L)),
\end{equation}
for every $p\in [n]^2$, and $h(p)$ is an integer for every $p\in [n]^2$ such that $V(G \uparrow p)\cap (K\cap L)=\emptyset$.
Note that for $i_1\in R$, $i_2\in C$ and $1\le j_1,j_2\le n$, the paths $G \uparrow(i_1,j_1)$ and $G \uparrow(i_1,j_2)$ are $g$-homotopic,
and the paths $G \uparrow(j_1,i_2)$ and $G \uparrow(j_2,i_2)$ are $g$-homotopic, since $R_{i_1}$ and $C_{i_2}$ are near-triangulations and $g$
is holomorphic on them.
It follows that $G \uparrow p_1$ and $G \uparrow p_2$ are $g$-homotopic for all $p_1,p_2 \in S$. Thus $h$ is constant on $S$ by Lemma~\ref{lem:homotopic}.
Let $H\in\bb{Z}$ be its value. 

Note that $|[n]\setminus C|\le t+1$.
Let us fix an element $c\in C$ and consider any $i\in [n]$.
Let $t_i=|V(R_i)\cap (K\cap L)|$.  For any $j\in [n]\setminus C$, the path $G \uparrow (i,j)$ is $(g,\Delta t_i)$-almost homotopic to the path $G \uparrow (i,c)$,
and by Lemma~\ref{lem:homotopic} we have $|h(i,j) - H|  \leq \Delta t_i$.  Thus

\begin{equation}\label{eq:5}
\left|\left(\sum_{j \in [n]} h(i,j)\right) - Hn\right|\le \sum_{j \in [n]} |h(i,j) - H| \leq \Delta t_i(t+1).
\end{equation}
Summing (\ref{eq:5}) over $i \in [n]$ and applying (\ref{eq:4}), we obtain
$$|\lambda(K\setminus L) - Hn^2| = \left|\sum_{p \in [n]^2} h(p) - Hn^2\right| \leq \Delta (t+1)^2.$$
However, by (\ref{eq:2.5}), $\lambda(K\setminus L)$ differs from an integer multiple of $n^2$ by at least $n^2/3$.
It follows that $t \geq \frac{n}{\sqrt{3\Delta}}-1$, as desired.
\end{proof}

\section{Partitions of a cube}\label{sec-main2}

Consider the $N\times N\times N$ grid $Q_N$ with non-decreasing diagonals.  A path in $Q_N$ with vertices
$v_1=(x_1,y_1,z_1),v_2=(x_2,y_2,z_2),\ldots,v_k=(x_k,y_k,z_k)$ in order
is \emph{a staircase from $v_1$ to $v_k$} if $x_i+1 = x_{i+1}$, $y_i\le y_{i+1}$ and $z_i\le z_{i+1}$ for $1\le i\le k-1$.
The \emph{$b$-square $\square(v,b)$} around a vertex $v=(x,y,z)\in V(Q_N)$ is defined by $$\square(v,B) = \{(x,y+d_y,z+d_z):0\le d_y,d_z\le b\}.$$
The \emph{$b$-enlargement} $B(b,P)$ of a staircase $P$ is the subgraph of $Q_N$ induced by
$\bigcup_{v \in V(S)}\square(v,b)$.  The \emph{left side} of the $b$-enlargement of a staircase from $u$ to $v$ is $\square(u,b)$
and the \emph{right side} is $\square(v,b)$.  For each edge $yz$ of the $b$-enlargement, let $W_{yz}$ be the closed walk consisting
of the path parallel to the staircase from the left side to $y$, the edge $yz$, the path parallel to the staircase from $z$ to the left side,
and possibly an edge in the left side.

We need the following property (another proof of a similar statement is implicit in~\cite{matopriv}, Proposition 3.1).

\begin{lemma}\label{lemma-connect}
Let $b \ge 0$ be an integer, and let $G$ be the $b$-enlargement of a staircase in $Q_N$. 
If $X\subseteq V(G)$ is a minimal set that intersects
every path from the left side of $G$ to the right side, then $G[X]$ is connected.
\end{lemma}
\begin{proof}
Let $S_1$ and $S_2$ be the left and the right side of $G$, respectively.
Since we can extend the staircase if necessary, we can assume that $X$ is disjoint from $S_1\cup S_2$.
Let $(A,B)$ be a separation of $G$ such that $S_1\subseteq A$, $S_2\subseteq B$ and $X=A\cap B$.
Let $l\in S_1$ and $r\in S_2$ be arbitrary vertices of the sides of $G$.

Consider any two vertices $x,y\in X$, and let $C$ be the vertex set of the component of $G[X]$ containing $x$.
Let $f:V(G)\to \mc{L}$ be defined by 
$$f(v)=\begin{cases} -1, & \text{if } v \in A\setminus B, \\
 0, &  \text{if } v \in C,\\
 \star, &  \text{if } v \in X\setminus C,\\
 1, &  \text{if } v \in B \setminus A.
\end{cases}$$
By the minimality of $X$, there exist paths $P_x$ and $P_y$ from $l$ to $r$ in $G$
such that $V(P_x)\cap X=\{x\}$ and $V(P_y)\cap X=\{y\}$.  Let $W$ be the closed walk consisting of $P_x$ and the reverse of $P_y$.
Note that $I_W=\sum_{i=1}^n I_{T_i}$ for some triangles $T_1$, \ldots, $T_n$ in $G$ (since $W$ is the sum of walks $W_{yz}$
for edges $yz$ of $W$ together with a closed walk in the left side of $G$, it suffices to observe that this claim holds for the walks $W_{yz}$ and for closed walks in the left side).

Since $f$ is holomorphic, $W$ is $(f,0)$-almost contractible, and thus $P_x$ and $P_y$ are $f$-homotopic.
By Lemma~\ref{lem:homotopic}, we have $\int_{P_x} df=\int_{P_y} df$.  Since $f$ is entire on $P_x$, we have
$\int_{P_x} df=2$, and thus $f(y)\neq\star$ (as otherwise we would have $\int_{P_y} df=0$).  Therefore,
$y$ is in the same component of $G[X]$ as $x$.

We conclude that $G[X]$ is connected.
\end{proof}

Let $A_1,A_2$ be a partition of vertices of $Q_N$.  Let $b\ge 0$ and $i\in\{1,2\}$ be integers, and  let $P$ be a staircase.
We say that $P$ is \emph{$(b,i)$-blocked} if every path in the $b$-enlargement of $P$ joining its left side $S_1$ with its right side $S_2$
intersects $A_i\setminus (S_1\cup S_2)$.

\begin{lemma}\label{lemma-permpair}
Let $A_1,A_2$ be a partition of vertices of $Q_N$.  Let $b\ge 0$ and $i\in\{1,2\}$ be integers.
Let $P$ be a $(b,i)$-blocked staircase, let $M$ be the $(b+1)$-enlargement of $P$, and
let $M_i$ be the subgraph of $M$ induced by $A_i\cap V(M)$.
There exists a connected component of $M_i$ containing all paths in $M_i$ which join the left side of $M$ with the right side.
\end{lemma}
\begin{proof}
Let $M_0$ be the $b$-enlargement of $P$.
Let $X$ be a minimal subset of $A_i\cap V(M_0)$ such that every path in $M_0$ from the left side to the right side intersects $X$.
By Lemma~\ref{lemma-connect}, $M_0[X]$ is connected.

Let $P'$ be a path in $M_i$ joining the left side of $M$ with the right side.
For a vertex $u=(x,y,z) \in M$, let $v=(x,y_0,z_0) \in V(P)$ be the unique
vertex of $P$ such that $u \in \square(v,b+1)$. Let $\pi(u)$ be the point of
$\square(v,b)$ closest to $u$ in the Euclidean distance. That is, $\pi(u)=(x,\min(y,y_0+b),\min(z,z_0+b))$.
Clearly $u$ is adjacent to $\pi(u)$. Moreover, it is easy to
check that if $u$ and $u'$ are adjacent then $\pi(u)$ and $\pi(u')$ are
adjacent or equal. 

It follows that the subgraph $\pi(P')$ of $M_0$ induced by $\{\pi(u): u \in V(P')\}$
is connected and contains vertices both in the left and in the right side
of $M_0$. Therefore, $\pi(P')$ intersects $X$, and thus $P'$ belongs to the
same component of $M_i$ as $X$, as desired.
\end{proof}

The \emph{$t \times t$-grid} is the graph whose vertex set consists of all pairs $\{(x,y) \in \bb{Z}^2 : 0 \leq x,y \leq t-1\}$
and two vertices $(x_1,y_1)$ and $(x_2, y_2)$ are adjacent iff $|x_1-x_2|+|y_1-y_2|=1$.
Let $P_1$ and $P_2$ be two  staircases.
A staircase $P$ \emph{joins} $P_1$ with $P_2$ if $P_1$ is the initial segment of $P$ and $P_2$ is the final segment of $P$, or vice versa.

Consider the $N\times N\times N$ grid $Q_N$.  For an integer $n\ge 1$ and a point $v=(x,y,z) \in V(Q_N)$ with $x,y,z\le N-n$,
let $Q_n(v)$ denote the subgrid of $Q_N$ induced by
vertices $(x',y',z') \in V(Q_N)$ such that $x \leq x' \leq x + n-1, y \leq y' \leq y + n-1,$ and $z \leq z' \leq z + n-1$.
For a positive integer $d$, let $p_d(j,k)=(4dj+4dk,2dj+dk,dj+2dk)$.
This definition is motivated by the following fact.

\begin{observation}\label{lemma-obsjoin}
Let $n\ge 1$, $b\ge 0$ and $d\ge n+b$ be integers.  For every vertex $z=(j,k)$ of the $(2t+1)\times (2t+1)$ grid,
let $p_z=p_d(j,k)$ and let $P_z$ be a staircase in $Q_n(p_z)$.
For every edge $uz$ of the $(2t+1)\times (2t+1)$ grid, there exists a staircase $P_{uz}$ joining $P_u$ with $P_z$.
Furthermore, the staircases can be chosen so that for any edges $uz$ and $u'z'$ with $\{u,z\}\cap \{u',z'\}=\emptyset$,
the $b$-enlargements of $P_{uz}$ and $P_{u'z'}$ are vertex-disjoint.
\end{observation}

In the following  proof of Theorem~\ref{thm-main2} we construct a subgraph of
$Q_n[A_1]$ or $Q_n[A_2]$ which closely resembles a subdivision of the $t \times t$ grid.
Unfortunately, obtaining an actual subdivision  seem to require a lot
more work, and therefore we content ourselves with obtaining the following less
structured certificate of large tree-width.  A collection of non-empty subsets $\mc{B}$
of the vertex set of a graph $G$ is called a \emph{bramble} if for all $B,B' \in \mc{B}$
the subgraph $G[B \cup B']$ of $G$ induced by $B \cup B'$ is
connected (and in particular, $G[B]$ is connected for every $B \in \mc{B}$). The
\emph{order} of $B$ is the minimum size of the set $S \subseteq V(G)$ such that
$S \cap B \neq \emptyset$ for every $B \in \mc{B}$. It is shown in~\cite{bramble}
that if $G$ contains a bramble of order $t$ then $G$ has tree-width at least
$t-1$.

We are now ready to complete the proof of Theorem~\ref{thm-main2}.

\begin{proof}[Proof of Theorem~\ref{thm-main2}]
We will prove by induction on $b$ that there exists $N=N(b)$ such that for
every partition $(A_1,A_2)$ of the vertex set of $Q_N$ and each $i\in \{1,2\}$, if the tree-width
of both $Q_N[A_1]$ and $Q_N[A_2]$ is less than $t$, then $Q_N$ contains a
$(b,i)$-blocked staircase.  Note that the theorem will
follow, by choosing $b\colonequals\lceil\sqrt{18}(t+1)\rceil-1$ and $n\colonequals N(b)$.
The $b$-enlargement $B$ of a staircase with sides $S_1$ and $S_2$
is a $(b+1,b+1)$-slab. As every path in $B$ from $S_1$ to $S_2$ intersects
$A_i\setminus (S_1 \cup S_2)$, Theorem~\ref{thm-main1} implies that the tree-width of
$Q_n[A_i]$ is at least $\frac{b+1}{\sqrt{18}}-1 \geq t$.

The base case $b=0$ is trivial with $N(0)=t+2$: we have $V(Q_t(1,1,1))\cap A_i\neq\emptyset$,
since $A_{3-i}$ induces a subgraph of tree-width less than $t$.  We move on to the induction step. Let
$n_0=N(b-1)$ and let $N=(8t+5)(n_0+b)$.
By the induction hypothesis, for every $v=(x,y,z)$ with $0\le x,y,z< N-n_0$, the subgrid $Q_{n_0}(v)$
contains a $(b-1,3-i)$-blocked staircase which we denote by $P_v$.

Let $d=n_0+b$, and for every vertex $z=(j,k)$ of the $(2t+1)\times (2t+1)$ grid,
let $p(j,k)=p_d(j,k)$.  Note that $Q_{n_0}(p(j,k))\subseteq Q_N$.
Let $P_z=P(p(j,k))$; by Lemma~\ref{lemma-permpair}, there exists a connected subgraph $M_z$
of the $b$-enlargement $B_z$ of $P_z$ such that $V(M_z)\subseteq A_{3-i}$ and $M_z$ contains all paths in $B_z[A_{3-i}]$ joining
the left side of $B_z$ with the right side.

For each edge $yz$ of the $(2t+1)\times (2t+1)$ grid, let $P_{yz}$ be the path as in Observation~\ref{lemma-obsjoin}.
If $P_{yz}$ is $(b,i)$-blocked then the proof of the induction step is finished.
Hence, we can assume that for every edge $yz$, there exists a path $R_{yz}$ with $V(R_{yz})\subseteq A_{3-i}$
in the $b$-enlargement of $P_{yz}$ joining the left side of this enlargement to the right side.
Note that $R_{yz}$ intersects $M_y$ and $M_z$.

For $0 \leq j \leq 2t$, let $S'_j$ be the path forming the $j$-th column of the $(2t+1)\times (2t+1)$ grid,
and let $$S_j=V\left(\bigcup_{y\in V(S'_j)} M_y\cup \bigcup_{yz\in E(S'_j)} R_{yz}\right).$$
Let the set $T_j$ be similarly defined for the $j$-th row of the grid.
As observed in the previous paragraph, $Q_N[S_j]$ and $Q_N[T_j]$ are
connected for all $0 \leq j\leq 2t$, and $M_{(j,k)} \subseteq S_j \cap T_k$.
Let $\mc{B}=\{S_j \cup T_j\}_{0 \leq j \leq 2t}$. Clearly $\mc{B}$ is a bramble
in $Q_N[A_{3-i}]$, and no vertex of $Q_N$ belongs to more than two elements of
$\mc{B}$ by Observation~\ref{lemma-obsjoin}. It follows that the order of
$\mc{B}$ is at least $t+1$, and thus $Q_N[A_{3-i}]$ has tree-width at least $t$,
yielding the desired contradiction.
\end{proof}

\noindent {\bf Acknowledgement.} This research was partially completed at a workshop held at the Bellairs Research Institute
in Barbados in April 2014. We thank the participants of the workshop, and especially Paul Seymour, for  helpful discussions.

\bibliographystyle{siam}
\bibliography{twd}

\end{document}